\documentclass{article}
\usepackage{amsmath}
\usepackage{amsthm}
\usepackage{amsfonts}
\usepackage{amssymb}
\usepackage{amscd}

\advance\voffset by 0.2 truein

\font\smallsc=cmcsc10
\font\smallsl=cmsl10

\newtheorem{Lem}{Lemma}[section]
\newtheorem{Prop}[Lem]{Proposition}
\newtheorem{Thm}[Lem]{Theorem}
\theoremstyle{definition}

\newcommand{\Fcal}{\mathcal{F}}
\newcommand{\Mcal}{\mathcal{M}}

\newcommand{\E}{\mathcal E}

\newcommand{\A}{\mathcal A}

\newcommand{\F}{\mathcal F}

\newcommand{\I}{\mathcal I}
\newcommand{\M}{\mathcal M}
\newcommand{\N}{\mathcal N}
\newcommand{\T}{\mathcal T}
\newcommand{\K}{\mathcal K}

\newcommand{\R}{\mathcal R}
\renewcommand{\L}{\mathcal L}
\renewcommand{\O}{\mathcal O}

\newcommand{\col}{\colon}

\newcommand{\ra}{\rightarrow}
\newcommand{\ol}{\overline}
\newcommand{\ul}{\underline}
\newcommand{\ox}{\otimes}
\newcommand{\IP}{\mathbb P}

\newcommand{\wh}{\widehat}
\newcommand{\J}{\ol J}
\newcommand{\Jcal}{\mathcal J}

\newcommand{\lra}{\longrightarrow}
\renewcommand{\:}{\colon}

\newcommand{\wt}{\widetilde}
\renewcommand{\P}{\mathcal P}
\newcommand{\rk}{\text{\rm rk}}

\def\cocoa{{\hbox{\rm C\kern-.13em o\kern-.07em C\kern-.13em o\kern-.15em A}}}

\begin{document}

\title{Semistable modifications of families of curves and compactified
Jacobians}
\author{Eduardo Esteves and Marco Pacini}
\maketitle

\begin{abstract} Given a family of nodal curves, a semistable
  modification of it is another family made up of curves obtained by 
inserting chains of rational curves of any given length at certain
nodes of certain curves of the original family. We give comparison
theorems 
between torsion-free, rank-1 sheaves in the former family and
invertible sheaves in the latter. We apply them to show that there are
functorial isomorphisms between the compactifications of relative
Jacobians of families of nodal curves constructed through Caporaso's
approach and those constructed through Pandharipande's approach.  
\noindent
\end{abstract}

\section{Introduction}

Compactifications of (generalized) Jacobians of 
(reduced, connected, projective) 
curves have been 
considered by several authors. Igusa \cite{Ig} was likely the first to study the 
degeneration of Jacobians of smooth curves when these 
specialize to nodal curves. Later, Mayer and Mumford \cite{MM} 
suggested realizing Igusa's 
degenerations using torsion-free, rank-1 sheaves to represent 
boundary points of compactifications of Jacobians. This was 
carried out by D'Souza \cite{DS} for irreducible curves, 
and by Oda and Seshadri \cite{OS} for reducible, nodal curves. In full 
generality, moduli spaces for torsion-free, rank-1 
(and also higher rank) sheaves on curves 
were constructed by Seshadri \cite{Sesh}.

As far as families are concerned, Altman and Kleiman 
\cite{AK1}, \cite{AK2}, \cite{AK3}, constructed relative 
compactifications of Jacobians 
for families of irreducible curves (and also higher dimension varieties). 
The author \cite{E01} continued their work, considering relative 
compactifications for any family of curves. The most general work in this 
respect is that of Simpson's \cite{simpson}, who constructed moduli 
spaces of coherent sheaves for families of schemes. 

It is also natural to ask whether a compactification of the relative Jacobian 
can be constructed over the moduli space $\ol M_g$ 
of Deligne--Mumford stable curves of genus $g$, for any $g\geq 2$. 
This is not a direct consequence of the works cited above for families, 
as there is no universal family of curves over $\ol M_g$. Such a 
compactification was constructed by Caporaso \cite{C}.

Caporaso's compactification represented a departure from the approach 
suggested by Mayer and Mumford, as the boundary points did not correspond 
to torsion-free, rank-1 sheaves on stable curves, but rather invertible 
sheaves on semistable curves of a special type, called quasistable 
curves,
where the exceptional components are isolated. 
The connection with the then usual approach was 
established one year later 
by Pandharipande \cite{Pand}, who 
constructed a compactification of the relative Jacobian (and also 
moduli spaces of vector bundles of any rank) over $\ol M_g$ using 
torsion-free, rank-1 sheaves, and showed that his compactification was 
isomorphic to Caporaso's.

More precisely, Caporaso produced a scheme $P^b_{d,g}$ 
coarsely representing the functor $\P^b_{d,g}$ 
that associates to each scheme $S$ the set of isomorphism classes of 
pairs $(Y/S,\L)$ of a family $Y/S$ of quasistable curves of 
(arithmetic) genus $g$, and an invertible sheaf $\L$ on $Y$ whose 
restrictions to the fibers of 
$Y/S$ have degree $d$ and satisfy certain ``balancing'' conditions; 
see Section~\ref{fi}.
On the other hand, Pandharipande produced a scheme 
$J^{ss}_{d,g}$ coarsely representing the functor $\Jcal^{ss}_{d,g}$ 
that associates to each scheme $S$ the set of isomorphism classes of 
pairs $(X/S,\I)$ of a family $X/S$ of stable curves of 
genus $g$, and a coherent $S$-flat sheaf $\I$ on $X$ whose 
restrictions to the fibers of 
$X/S$ are torsion-free, rank-1 sheaves of degree $d$ satisfying certain 
``semistability'' conditions; see Section~\ref{fi}.

Essentially, Pandharipande constructed in \cite{Pand}, 10.2, 
p.~465, a map of functors 
$\Phi^b\:\P^b_{d,g}\to\Jcal^{ss}_{d,g}$, and showed that the 
corresponding map of schemes $\phi\:P^b_{d,g}\to J^{ss}_{d,g}$ is bijective in 
10.3, p.~468. Then he 
used the normality of $J^{ss}_{d,g}$, which he had proved in Prop.~9.3.1, p.~464, 
to conclude that $\phi$ is an isomorphism in Thm.~10.3.1, p.~470. 

In the present article we prove that $\Phi^b$ is itself an isomorphism 
of functors, our Theorem \ref{isofcor}, which thus entails that $\phi$ 
is an isomorphism. We do so by describing the inverse map. In fact, 
our Proposition \ref{biss} implies that $\Phi^b$ is the 
restriction of a map $\Phi\:\P_{d,g}\to\Jcal_{d,g}$ 
between ``larger'' functors, without the 
extra conditions of ``balancing'' and ``semistability.'' And our 
Theorem \ref{isof} claims that $\Phi$ is an isomorphism, describing its 
inverse.

We feel that these results are of interest, not only because they give another 
proof of the existence of the isomorphism $\phi$, but also because of the 
immediate application to stacks. The point of view of stacks was 
applied to the problem of compactifying the relative Jacobian over 
$\ol M_g$ in \cite{TJ} and in \cite{CAJM}, 
the latter in the special situation where 
Deligne--Mumford stacks arise, and in more generality in \cite{melo1}. 
See \cite{melo2} as well, for compactifications over the stacks of pointed 
stable curves. It is a natural point of view, and should be further 
studied. We give here a small contribution to this study. 

Furthermore, we go beyond showing that $\Phi^b$ is an isomorphism. 
More generally, we study families of semistable curves, and give
comparison theorems between torsion-free rank-1 sheaves on nodal curves
and invertible sheaves on semistable modifications of them; see
Theorems \ref{famchain}, \ref{compadm} and \ref{famchain2}. 
Though technical, we believe 
these are useful theorems to have when working in the field. Indeed,
they have already proved fundamental in our study of Abel maps; see 
\cite{CEP}, from which \cite{ACP}, \cite{CP}, \cite{Pac1} and \cite{Pac2} derive. 
In \cite{CEP} we study the construction of degree-2 Abel maps for nodal curves, and we need to deal with invertible sheaves on semistable curves containing chains of two exceptional components; it is expected that longer chains will occur in the study of higher degree Abel maps.

Some of the results in these notes may be well-known to the specialists. 
For instance, Propositions \ref{L2I} and \ref{I2L}, are essentially stated in 
\cite{CCaCo}, Prop.~4.2.2, p.~3754, for whose proof the reader is
mostly 
refered to 
\cite{Falt} and \cite{TJ0}. However, detailed statements and proofs
are given here, together with generalizations and a more global
approach, which works over general base schemes. 

In short, here is how the paper is structured. In Section~\ref{2} we
describe our basic objects, torsion-free, rank-1 sheaves on families
of curves, present the notion of stability, and give cohomological
characterizations for the existence of certain inequalities for
degrees of invertible sheaves on chains of rational curves. 

In 
Section~\ref{3},  we prove our main result, Theorem~\ref{famchain},
which gives necessary and sufficient 
conditions under which the pushforward $\psi_*\L$ of an
invertible sheaf $\L$ under a map of curves $\psi\:Y\to X$
contracting exceptional components is torsion-free, rank-1. We give as
well sufficient conditions for when two invertible sheaves 
have the same pushforward. 
In Section~\ref{4}, we prove 
Theorem~\ref{famchain2}, which compares the various notions of
stability for $\L$ with those for $\psi_*\L$. In Section~\ref{5}, we
apply these theorems in the special situation where the exceptional
components of $Y$ are isolated. In addition, we show how to do the opposite 
construction, that is, how to get an invertible sheaf $\L$ on $Y$ 
from a torsion-free, rank-1 sheaf $\I$ on $X$ in such a way that 
$\I=\psi_*\L$. All the constructions apply to families of
curves. Then, in Section \ref{fi}, we apply them to produce an inverse
to $\Phi^b$. 

\section{Sheaves on curves}\label{2}

\subsection{Curves} 

A \emph{curve} is a reduced, connected, projective scheme of pure dimension 1 
over an algebraically closed field $K$. A curve may have several 
irreducible components, which will be simply called \emph{components}. 
We will always assume our curves to be \emph{nodal}, 
meaning that the singularities are \emph{nodes}, that is, 
analytically like the origin on the 
union of the coordinate axes of the plane $\mathbb A^2_K$. 

We say that a curve $X$ has \emph{genus} $g$ if 
$h^1(X,\O_X)=g$. This is in fact the so-called arithmetic genus, but the 
geometric genus will play no role here.

A \emph{subcurve} of a curve $X$ is the reduced union of a nonempty 
collection of its components. A subcurve is a curve if and 
only if it is connected. 
Given a proper subcurve $Y$ of $X$, we 
will let $Y'$ denote the \emph{complementary subcurve}, that is, the 
reduced union of the remaining components of $X$. Also, 
we let $k_Y$ denote the number of points of $Y\cap Y'$. 
Since $X$ is connected, $k_Y\geq 1$. A component $E$ of a curve $X$ is called 
\emph{exceptional} if $E$ is smooth, rational, $E\neq X$ and $k_E\leq 2$. 

We will call a curve $X$ \emph{semistable} 
if all exceptional components $E$ have $k_E=2$; \emph{quasistable} 
if, in addition, no two exceptional components meet; and \emph{stable} 
if there are no exceptional components.

A \emph{chain of rational curves} is a 
curve whose components are smooth and
rational and can be ordered, $E_1,\dots,E_n$, in such a way that 
$\#E_i\cap E_{i+1}=1$ for $i=1,\dots,n-1$ and  
$E_i\cap E_j=\emptyset$ if $|i-j|>1$.  If $n$ is the number of components, 
we say that the chain has \emph{length $n$}. Two chains of 
the same length are isomorphic. The
components $E_1$ and $E_n$ are called the \emph{extreme curves} of the 
chain. A connected subcurve of a chain is also a chain, and is called a 
\emph{subchain}.

Let $\N$ be a collection of nodes of a curve $X$, 
and $\eta\:\N\to\mathbb N$ a function. Denote by $\wt X_\N$ the partial 
normalization of $X$ along $\N$. For each $P\in\N$, let $E_P$ be a chain of 
rational curves of length $\eta(P)$. Let $X_\eta$ denote the curve 
obtained as the union of $\widetilde X_\N$ and the $E_P$ for $P\in\N$ in the 
following way: Each chain $E_P$ intersects no other chain, but intersects 
$\wt X_\N$ transversally at two points, the branches over $P$ on $\wt X_\N$ 
on one hand, and nonsingular points on each of the two extreme curves of $E_P$  
on the other hand. There is a natural map 
$\mu_\eta\:X_\eta\to X$ collapsing each 
chain $E_P$ to a point, whose restriction to $\wt X_\N$ is the partial 
normalization map. The curve $X_\eta$ and the map $\mu_\eta$ are well-defined 
up to $X$-isomorphism.

All schemes are assumed locally Noetherian. 
A point $s$ of a scheme $S$ is a map $\text{Spec}(K)\to S$, where $K$ is a 
field, denoted $\kappa(s)$. If $\kappa(s)$ is algebraically closed, we say that $s$ is 
geometric. 

A \emph{family of} (\emph{connected}) \emph{curves} 
is a proper and flat morphism 
$f\colon X\rightarrow S$ whose geometric fibers are
connected curves. If $s$ is a geometric point of $S$, put
$X_s:=f^{-1}(s)$. If $T$ is a $S$-scheme, put $X_T:=X\times_S T$; the
second projection $X_T\to T$ is also a family of curves.

If all the geometric fibers of $f$ are 
(semistable, quasistable, stable) curves (of genus $g$), 
we will say that 
$f$ or $X/S$ is a 
\emph{family of} 
(\emph{semistable, quasistable, stable}) 
\emph{curves} 
(\emph{of genus $g$}). 

If $X$ is a curve over an algebraically closed field $K$, a 
\emph{regular smoothing} of $X$ is the data $(f,\xi)$ consisting 
of a generically smooth family of curves $f\:Y\to S$, where $Y$ is regular
and $S$ is affine with ring of functions 
$K[[t]]$, the ring of formal power series over $K$, and an 
isomorphism $\xi\:X\to Y_0$, where $Y_0$ is the special fiber of
$f$. A \emph{twister}  of $X$ is an invertible sheaf on $X$ of the
form  $\xi^*\O_Y(Z)|_{Y_0}$, where $(f\:Y\to S,\xi)$ 
is a regular smoothing of $X$, and $Z$ is a Cartier divisor of $Y$ 
supported in $Y_0$, so a formal sum of 
components of $Y_0$. 
 A twister has degree 0 by continuity of the 
degree, since $\O_Y(Z)$ is trivial away from $Y_0$. 

If $Z$ is a formal
sum of the components of $X$, we define 
$$
\O_X(Z):=\xi^*\O_Y(\xi(Z))|_{Y_0}.
$$
This definition depends on the
choices of $f$ and $\xi$. However, for our purposes here, the definition is
good enough as it is.

\subsection{Sheaves}\label{sect2.2}

Let $f\col X\ra S$ be a family of curves.
Given a coherent sheaf $\Fcal$ on $X$ and a geometric 
point $s$ of $S$, we will 
let $\Fcal_s:=\Fcal|_{X_s}$.
More generally, given any $S$-scheme $T$,
denote by $\Fcal_T$ the
pullback of $\Fcal$ to $X_T$
under the first projection $X_T\to X$.

 Let $\I$ be a
$S$-flat coherent sheaf on $X$. We say that
$\I$ is \emph{torsion-free} on $X/S$ if,
for each geometric point $s$ of $S$, the
associated points of $\I_s$
are generic points of $X_s$. We say that
$\I$ is \emph{of rank $1$} or \emph{rank-$1$} on $X/S$
if, for each geometric point $s$ of $S$, the
sheaf $\I_s$ is invertible
on a dense open subset of $X_s$. 
We say that $\I$ is \emph{simple} on $X/S$ if, for each geometric point $s$ of $S$, we have 
$\text{Hom}(\I_s,\I_s)=\kappa(s)$.

Since $X$ is flat over $S$, with reduced and connected fibers,
each invertible sheaf on $X$ is torsion-free, rank-1 and simple
on $X/S$. In particular, so is the relative dualizing sheaf 
of $X/S$.

We say that $\I$ has degree $d$ on $X/S$
if $\I_s$ has degree $d$ for each $s\in S$, that is,
$$
d=\chi(\I_s)-\chi(\O_{X_s})
$$
for each $s\in S$.

Given a geometric point $s$ of $S$ and a subcurve $Y$ of $X_s$, 
let $\I_Y$ denote the restriction of
$\I$ to $Y$ modulo torsion. In other words,
if $\xi_1,\dots,\xi_m$ are the generic points of $Y$,
let $\I_Y$ denote the image of the natural map
$$
\I|_Y\lra\bigoplus_{i=1}^m  (\I|_Y)_{\xi_i}.
$$
Also, let $\deg_Y(\I)$ denote the degree of
$\I_Y$, i.e.
$$
\deg_Y(\I):=\chi(\I_Y)-\chi(\O_Y).
$$

Let $X$ be a (connected) curve over an algebraically closed field $K$
and 
denote by $X_1,\ldots,X_p$ its 
components.
Fix an integer $d$. 
Since $X$ is a proper scheme over $K$, by \cite{BLR}, Thm.~8.2.3,
p.~211, there is a scheme, locally of finite type over $K$, parameterizing 
degree-$d$ invertible sheaves on $X$; denote it by $J_X^d$. It decomposes as 
\begin{equation}\label{Jacobdecomp}
J^d_X=\underset{ d_1+\ldots+d_p=d}{\underset{\ul{d}=(d_1,\dots,d_p)}
{\coprod}}J^{\ul{d}}_X,
\end{equation}
where $J^{\ul{d}}_X$ is the connected component of $J_X^d$ parameterizing invertible sheaves $\L$ such that 
$\deg(\L|_{X_i})=d_i$  for $i=1,\dots,p$. The $J^{\ul{d}}_X$ are 
quasiprojective varieties. 

The scheme $J_X^d$ is in a natural way an open subscheme of $\J_X^d$, 
the scheme over $K$ 
parameterizing torsion-free, rank-1, simple sheaves of degree $d$ 
on $X$; see \cite{E01} for the construction of $\J_X^d$ and its properties. 
 The scheme $\J_X^d$ is universally closed over $K$ 
but, in general, not separated and only locally of finite type. Moreover, in 
contrast to $J_X^d$, the scheme $\J_X^d$ is connected, hence not
easily decomposable. Thus, to deal with a 
manageable piece of it, we resort to polarizations. 

Let $\E$ be a locally free sheaf on $X$ of constant rank, and 
$\I$ a torsion-free, rank-1 sheaf on $X$. 
We say that $\I$ 
is \emph{semistable} (resp.~\emph{stable}, resp.~\emph{$X_i$-quasistable}) 
\emph{with respect to $\E$} if
\begin{enumerate}
\item $\chi(\I\ox \E)=0$,
\item $\chi(\I_Y\ox \E|_Y)\geq 0$ for each proper subcurve 
$Y\subset X$ (resp.~with equality never, resp. 
with equality only if $X_i\not\subseteq Y$).
\end{enumerate}
Notice that it is enough to 
check Property 2 above for connected subcurves $Y$. Also, Property 1 is equivalent to the numerical condition that
$$
\rk(\E)(\deg(\I)+\chi(\O_X))+\deg(\E)=0.
$$

The $X_i$-quasistable sheaves are simple, what can be easily proved using 
for instance \cite{E01}, Prop.~1, p.~3049. 
Their importance is that they form 
an open subscheme $\J^{\E,i}_X$ of $\J^d_X$ that is projective over $K$. 

Let $f\:X\to S$ be a family of curves. 
Let $\E$ be a locally free sheaf on $X$ of constant rank and 
$\I$ a torsion-free, rank-1 sheaf on $X/S$. Let 
$\sigma\:S\to X$ 
be a section of $f$ through its smooth locus. 
We say that $\I$ 
is \emph{semistable} (resp.~\emph{stable}, resp.~\emph{$\sigma$-quasistable}) 
\emph{with respect to $\E$} if, for each 
geometric point $s$ of $S$, the sheaf $\I_s$ is semistable
(resp.~stable, resp.~$X_{s,\sigma}$-quasistable) with respect to
$\E_s$. Here, $X_{s,\sigma}$ is the component of $X_s$ containing 
$\sigma(s)$.

There is an algebraic space $\ol J_{X/S}$ parameterizing 
torsion-free, rank-1, simple sheaves on $X/S$, containing the locus 
$J_{X/S}$ parameterizing invertible sheaves as an open
subset. Remarkable facts are that, first, up to an \'etale base
change, $\ol J_{X/S}$ is a scheme; second, the locus of $\ol J_{X/S}$
parameterizing the sheaves on $X/S$ which are $\sigma$-quasistable 
with respect to $\E$ is an open subspace which is proper over $S$.
 
\subsection{Chains of rational curves} 

If $E$ is a chain of rational curves and $\L$ is an invertible sheaf on 
$E$, then $\L$ is determined by its restrictions to the components of $E$, and thus by the degrees of these restrictions. In particular, 
$\L\cong\O_E$ if and only if $\deg(\L|_F)=0$ for each component $F\subseteq E$. Also, $\L$ is the dualizing sheaf of $E$ if 
its degree on each component is zero, but for the extreme curves, where the degree is $-1$.

\begin{Lem}\label{chainh1h0} Let $E$ be a chain of rational curves of length 
$n$. Let $E_1$ and $E_n$ denote the extreme curves. 
Let $\L$ be an invertible sheaf on $E$.
Then the following statements hold:
\begin{enumerate}
\item $\deg(\L|_F)\geq -1$ for every subchain $F\subseteq E$ if and
  only if $h^1(E,\L)=0$.
\item $\deg(\L|_F)\leq 1$ for every subchain $F\subseteq E$ if and
  only if
$$
h^0(E,\L(-P-Q))=0
$$
for any two points $P\in E_1$ and $Q\in E_n$ on the nonsingular locus of $E$.
\end{enumerate}
\end{Lem}

\begin{proof}  Let $E_1,\dots,E_n$ be the components of $E$, 
ordered in such a way that $\# E_i\cap E_{i+1}=1$ for $i=1,\dots,n-1$. 
We prove the statements by induction on $n$. If $n=1$ all the 
statements follow from the knowledge of the cohomology of 
the sheaves $\O_{\IP^1_K}(j)$. 

Suppose $n>1$. We show Statement 1. Assume that $\deg(\L|_F)\geq -1$ for every 
subchain $F\subseteq E$. Consider the natural exact sequence
$$
0\to\L|_{E_1}(-N)\to\L\to\L|_{E'_1}\to0,
$$
where $E'_1:=\ol{E-E_1}$ and $N$ is the unique point of $E_1\cap E'_1$. 
By induction, $h^1(E'_1,\L|_{E'_1})=0$. If $\deg\L|_{E_1}\geq 0$ then 
$h^1(E_1,\L|_{E_1}(-N))=0$ as well, and hence
$h^1(E,\L)=0$ from the long exact sequence in cohomology.

Suppose now that $\deg\L|_{E_1}<0$. If $\deg\L|_{E_n}\geq 0$, we 
invert the ordering of the 
chain, and proceed as above. Thus we may suppose $\deg\L|_{E_n}<0$ as well. 
Since $\deg\L|_E\geq -1$, there is $i\in\{2,\dots,n-1\}$ such that
$\deg\L|_{E_i}\geq 1$. Let $F_1:=E_1\cup\cdots\cup E_{i-1}$ and 
$F_2:=E_{i+1}\cup\cdots\cup E_n$.
Consider the natural exact sequence
$$
0\to\L|_{E_i}(-N_1-N_2)\to\L\to\L|_{F_1}\oplus\L|_{F_2}\to0,
$$
where $N_1$ and $N_2$ are the two points of intersection of $E_i$ with 
$E'_i:=\ol{E-E_i}$. By induction,
$h^1(F_1,\L|_{F_1})=h^1(F_2,\L|_{F_2})=0$. Also, since 
$\deg \L|_{E_i}\geq 1$, we have $h^1(E_i,\L|_{E_i}(-N_1-N_2))=0$, and thus 
it follows from the long
exact sequence in cohomology that $h^1(E,\L)=0$ as well.

Assume now that $h^1(E,\L)=0$. Then $h^1(F,\L|_F)=0$ for every subchain
$F\subseteq E$. By induction, $\deg(\L|_F)\geq -1$ for every proper
subchain $F\subsetneq E$. Since $E$ is the union of two proper
subchains, it follows that $\deg(\L)\geq -2$. Assume by contradiction
that $\deg(\L)=-2$. Then $\deg(\L|_F)=-1$ for every proper subchain 
$F\subsetneq E$ containing $E_1$ or $E_n$. It follows that 
$$
\deg(\L|_{E_i})=\begin{cases}0&\text{if $1<i<n$},\\
-1&\text{otherwise.}\end{cases}
$$
But then $\L$ is the dualizing sheaf of $E$, and thus $h^1(E,\L)=1$,
reaching a contradiction. The proof of Statement 1 is complete.

Statement 2 is proved in a similar way. Alternatively, it is enough 
to observe that $\O_E(-P-Q)$ is the dualizing sheaf of $E$, 
and thus, by Serre Duality,
$$
h^0(E,\L(-P-Q))=h^1(E,\L^{-1}).
$$
So Statement 2 follows from 1.
\end{proof}

\section{Admissibility}\label{3}

Let $f\col X\ra S$ be a family of curves. 
Let $\psi\:Y\to X$ be a proper 
morphism such that the composition $f\psi$ is another family 
of curves. We say that $\psi$ is a \emph{semistable modification} of $f$ 
if for each geometric point $s$ of $S$ there are 
a collection of nodes $\N_s$ 
of $X_s$ and a map
$\eta_s\:\N_s\to\mathbb N$ such that the induced map 
$\psi_s\col Y_s\to X_s$ 
is 
$X_s$-isomorphic to $\mu_{\eta_s}\:(X_s)_{\eta_s}\to X_s$. If $\eta_s$
is constant and equal to 1 for every $s$, we say that $\psi$ is a 
\emph{small semistable modification} of $f$. 
\marginpar{\bf New definition}

Assume $\psi$ is a semistable modification of $f$. Let $\L$ be an 
invertible sheaf on $Y$.  We say that $\L$ is 
$\psi$-\emph{admissible} (resp.~\emph{negatively $\psi$-admissible}, 
resp.~\emph{positively $\psi$-admissible}, 
resp.~$\psi$-\emph{invertible}) at a given geometric point $s$ of $S$ 
if the restriction of $\L$ to every chain of rational curves of $Y_s$
over a node of $X_s$ has degree $-1$, $0$ or $1$ 
(resp.~$-1$ or $0$, resp.~$0$ or $1$, resp.~$0$). We say that $\L$ is 
$\psi$-\emph{admissible} (resp.~\emph{negatively $\psi$-admissible}, 
resp.~\emph{positively $\psi$-admissible}, 
resp.~$\psi$-\emph{invertible}) if $\L$ is so at every $s$. 
Notice that, if $\L$ is negatively (resp.~positively) 
$\psi$-admissible, for every 
chain of rational curves of $Y_s$ over a node of $X_s$, 
the degree of $\L$ on each component of the chain is 0 but 
for at most one component where the degree is $-1$ (resp.~$1$).

\begin{Thm}\label{famchain}
Let $f\:X\to S$ be a family of curves and $\psi\:Y\to X$ a 
semistable modification of $f$. 
Let $\L$ be an invertible sheaf on $Y$ of relative degree $d$ over $S$. 
Then the following statements hold:
\begin{enumerate}
\item The points $s$ of $S$ at which $\L$ is
  $\psi$-admissible (resp.~negatively $\psi$-admissible,
  resp.~positively $\psi$-admissible, 
resp.~$\psi$-invertible) form an open subset of $S$.
\item
$\L$ is $\psi$-admissible if and only if 
$\psi_*\L$ is a torsion-free, rank-$1$ sheaf on $X/S$ of 
relative degree $d$, whose formation commutes with base change. In
this case, $R^1\psi_*\L=0$.
\item If $\L$ is $\psi$-admissible then the evaluation
map $v\:\psi^*\psi_*\L\to\L$ is surjective if and only if $\L$ is 
positively $\psi$-admissible. Furthermore, $v$ is bijective if and
only if $\L$ is $\psi$-invertible, if and only if $\psi_*\L$ is invertible.
\end{enumerate}
\end{Thm}

\begin{proof} All of the statements and hypotheses 
are local with respect to the \'etale topology of $S$. So we may
assume $S$ is Noetherian and that there is an 
invertible sheaf $\A$ on $X$ that is relatively ample over $S$.  Let 
$\widehat{\A}:=\psi^*\A$. 

We prove Statement 1 first. 
For each geometric point $s$ of $S$, let $E_s$ be the subcurve of
$Y_s$ which is the union of all the components contracted by $\psi_s$,
and let $\wt X_s$ be the partial
normalization of $X_s$ obtained as the union of the remaining
components. Since $\psi|_{\widetilde X_s}\:\widetilde X_s\to X_s$ is a 
finite map, it follows that $\widehat\A|_{\wt X_s}$ is ample, and thus 
$$
h^1(\wt X_s,(\L\ox\widehat{\A}^{\ox m_s})|_{\wt X_s}(-\textstyle\sum P_i))=0
$$
for every large enough integer $m_s$, where the sum runs over all the branch points of 
$\wt  X_s$ above $X_s$.  Since $S$ is Noetherian, a large enough
integer works for all $s$, that is, for every $m>>0$,
\begin{equation}\label{hm>0}
h^1(\wt X_s,(\L\ox\widehat{\A}^{\ox m})|_{\wt X_s}
(-\textstyle\sum
P_i))=0\quad\text{for each geometric point $s$ of $S$}.
\end{equation}

Now, for each integer $m$ consider the natural exact sequence
\begin{equation}\label{exseq}
0\lra(\L\ox\widehat{\A}^{\ox m})|_{\wt X_s}(-\textstyle\sum P_i)\lra
\L_s\ox\widehat{\A}_s^{\ox m}\lra(\L\ox\widehat{\A}^{\ox m})|_{E_s}\lra 0
\end{equation}
and its associated long exact sequence in cohomology. If $m$ is large enough that \eqref{hm>0}
holds, then
\begin{equation}\label{h1E00}
h^1(Y_s,\L_s\ox\widehat{\A}_s^{\ox m})=
h^1(E_s,\L\ox\widehat{\A}^{\ox m}|_{E_s}).
\end{equation}
On the other hand, since $\widehat{\A}$ is a pullback from $X$, it follows
that 
\begin{equation}\label{h1E0}
h^1(E_s,\L\ox\widehat{\A}^{\ox m}|_{E_s}) =
\textstyle\sum_F h^1(F,\L|_F)\quad
\text{for every integer $m$},
\end{equation}
where the sum runs over all the maximal chains $F$ of rational curves on
$Y_s$ contracted by $\psi_s$. Putting together \eqref{h1E00} and
\eqref{h1E0}, it follows now from Lemma~\ref{chainh1h0}
that 
\begin{equation}\label{adm=0}
h^1(Y_s,\L_s\ox\widehat{\A}_s^{\ox m})=0
\end{equation}
if and only if $\deg(\L|_F)\geq -1$ for every chain $F$ of
rational curves on $Y_s$ contracted by $\psi_s$. This is the case if
$\L$ is $\psi$-admissible at $s$.

It follows from semicontinuity of cohomology that the geometric points
$s$ of $S$ such that $\L_s$ has degree at least $-1$ on every chain
of rational curves of $Y_s$ contracted by $\psi_s$ 
form an open subset $S_1$ of $S$. Likewise, for
each integer $n$, the geometric points $s$ of $S$ such that
$\L^{\otimes n}_s$ has degree at least $-1$ on every chain of
rational curves of $Y_s$ contracted by $\psi_s$ form an open subset
$S_n$ of $S$. Then $S_1\cap S_{-1}$ parameterizes those $s$ for which
$\L_s$ is $\psi_s$-admissible, $S_1\cap S_{-2}$ parameterizes
those $s$ for which $\L_s$ is negatively $\psi_s$-admissible, 
$S_2\cap S_{-1}$ parameterizes
those $s$ for which $\L_s$ is positively $\psi_s$-admissible, 
and
$S_2\cap S_{-2}$ parameterizes those $s$ for which $\L_s$ is
$\psi_s$-invertible.

We prove Statement 2 now. 
Assume for the moment that $\L$ is $\psi$-admissible. 
To show that $\psi_*\L$ is flat over $S$, 
we need only show that $f_*(\psi_*\L\ox\A^{\ox m})$ is locally free for
each $m>>0$. By the projection formula, we need only show that
$g_*(\L\ox\widehat{\A}^{\ox m})$ is locally free for each $m>>0$,
where $g:=f\psi$. 
This follows from what we have already proved: For each large enough
integer $m$ such that \eqref{hm>0} holds, also 
\eqref{adm=0} holds for each geometric 
point $s$ of $S$, because $\L$ is $\psi$-admissible.

Furthermore, taking the long exact sequence in higher direct images of $\psi_s$ for 
the exact sequence \eqref{exseq} with $m=0$, using \eqref{h1E0} and 
that $\psi_s|_{\wt X_s}\:\wt X_s\to X_s$ is a finite map, it follows 
that $R^1\psi_{s*}(\L_s)=0$ for every geometric point $s$ of $S$. 
Since the fibers of $\psi$ have at most 
dimension 1, the formation of $R^1\psi_*(\L)$ commutes with base change, and 
thus $R^1\psi_*(\L)=0$.

Another consequence of \eqref{adm=0} holding for each geometric point
$s$ of $S$ is that the formation of 
$g_*(\L\ox\widehat{\A}^{\ox m})$ commutes with base
change for $m>>0$. 
We claim now that the base-change map 
$\lambda_X^*\psi_*\L\to\psi_{T*}\lambda_Y^*\L$ is an
isomorphism for each Cartesian diagram of maps
$$
\begin{CD}
Y_T @>\lambda_Y >> Y\\
@V\psi_TVV @V\psi VV\\
X_T @>\lambda_X >> X\\
@Vf_TVV @Vf VV\\
T @>\lambda >> S.
\end{CD}
$$
Indeed, since $\A$ is relatively ample over $S$, 
it is enough to check that the induced map
\begin{equation}\label{bch}
f_{T*}(\lambda_X^*\psi_*\L\ox\lambda_X^*\A^{\ox m})\lra
f_{T*}(\psi_{T*}\lambda_Y^*\L\ox\lambda_X^*\A^{\ox m})
\end{equation}
is an isomorphism for $m>>0$. But, by the projection formula,
the right-hand side is simply 
$f_{T*}\psi_{T*}\lambda_Y^*(\L\ox\wh\A^{\ox m})$.
Also, since $\psi_*\L$ is $S$-flat, 
the left-hand side is $\lambda^*f_*(\psi_*(\L)\ox\A^{\ox m})$ for $m>>0$, 
whence equal to 
$\lambda^*f_*\psi_*(\L\ox\widehat{\A}^{\ox m})$ by the projection formula. 
So, since the formation of $g_*(\L\ox\widehat{\A}^{\ox m})$ 
commutes with base
change for $m>>0$, 
it follows that \eqref{bch} is an isomorphism for $m>>0$, as asserted.

To prove the remainder of Statement 2 and Statement 3 we 
may now assume that $S$ is a geometric point. For Statement 2, 
we need only show now that $\psi_*\L$ is a torsion-free, rank-1
sheaf of degree $d$ on $X$ if and only if $\L$ is $\psi$-admissible. 
Let $F_1,\dots,F_e$ be the maximal chains of rational curves
of $Y$ contracted by $\psi$, to $P_1,\dots,P_e\in X$. Let $E$ be the
union of the $F_i$ and $\wt X$ the union of the remaining
components. For each $i=1,\dots,e$, let $P_{i,1},P_{i,2}\in Y$ be the
points of intersection between $F_i$ and $\wt X$.  
Taking higher direct images under $\psi$ in the 
natural exact sequences
\begin{equation}\label{exx}
\begin{aligned}
&0\to\L|_{\wt X}(-\textstyle\sum P_{i,j})\to\L\to\L|_E\to0,\\
&0\to\L|_E(-\textstyle\sum P_{i,j})\to\L\to\L|_{\wt X}\to0,
\end{aligned}
\end{equation} 
and using that $\psi|_{\wt X}$ is a finite map, we get 
\begin{equation}\label{RRR}
R^1\psi_*\L=R^1\psi_*\L|_E
\end{equation} 
and the exact sequence
$$
0\to\psi_*\L|_E(-\textstyle\sum P_{i,j})\to\psi_*\L\to\psi_*\L|_{\wt X}\to
R^1\psi_*\L|_E(-\textstyle\sum P_{i,j})\to R^1\psi_*\L\to 0 .
$$
Since $\psi|_{\widetilde X}$ is also birational, $\psi_*\L|_{\widetilde  X}$ is a torsion-free, 
rank-1 sheaf of degree $\deg\L|_{\widetilde  X}+e$. Since 
$\psi_*\L|_E(-\textstyle\sum P_i)$ is supported at finitely many
points, it follows that $\psi_*\L$ is
torsion-free if and only if $h^0(E,\L|_E(-\sum P_i)=0$. The latter holds if
and only if the degree of $\L$ on each chain of rational curves in $E$
is at most $1$, by Lemma \ref{chainh1h0}. Furthermore, if it
holds, then $R^1\psi_*\L|_E(-\textstyle\sum P_i)$ has length
$1-\deg\L|_{F_i}$ at each $P_i$ by the Riemann--Roch Theorem. Since 
$\deg\L|_{\wt  X}+\deg\L|_E=d$, it follows that $\deg\psi_*\L=d$ if
and only if $R^1\psi_*\L=0$. By \eqref{RRR}, the latter holds if and
only if $h^1(E,\L|_E)=0$, thus if and only if the degree 
of $\L$ on each chain of rational curves in $E$
is at least $-1$, by Lemma \ref{chainh1h0}. The proof of Statement 2
is complete.

Assume from now on that $\L$ is $\psi$-admissible. Then 
$\psi_*\L|_E(-\sum P_{i,j})=0$, and 
thus it follows from the exact sequences in \eqref{exx} that
\begin{equation}\label{ii}
\psi_*(\L|_{\wt X}(-\textstyle\sum P_{i,j}))\subseteq\psi_*\L\subseteq
\psi_*(\L|_{\wt X}).
\end{equation}
Furthermore, since $R^1\psi_*\L=0$ and since $R^1\psi_*\L|_E(-\sum
P_{i,j})$ is supported with length $1-\deg\L|_{F_i}$ at $P_i$, the
rightmost inclusion is strict at $P_i$ if $\deg\L|_{F_i}=0$, and an
equality if $\deg\L|_{F_i}=1$, for each $i=1,\dots,e$. 
In particular, if $\psi_*\L$ is
invertible, then $\deg\L|_{F_i}=0$ for every $i=1,\dots,e$.

Moreover, for each
$i=1,\dots,e$, we have
the following natural commutative diagram:
\begin{equation}\label{ddg}
\begin{CD}
\psi_*\L|_{P_i} @>v'_i>> \psi_*(\L|_{F_i})\\
@VVV @V(\rho_{i,1},\rho_{i,2})VV\\
\psi_*(\L|_{\wt X})|_{P_i} @>>> \psi_*(\L|_{P_{i,1}}\oplus\L|_{P_{i,2}})
\end{CD}
\end{equation}
where all the maps are induced by restriction. 
Then $\psi_*\L$ is invertible at $P_i$ if and
only if $\deg\L|_{F_i}=0$ and the 
compositions
\begin{equation}\label{cddg}
\psi_*\L\to\psi_*(\L|_{\wt X})\to\psi_*(\L|_{P_{i,j}})
\end{equation}
are nonzero for $j=1,2$. This is the case only if the maps 
$\rho_{i,1}$ and $\rho_{i,2}$ are nonzero. 

Now, if $\deg\L|_{F_i}=0$
then $\rho_{i,1}$ and $\rho_{i,2}$ are nonzero if and only if
$\L|_{F_i}=\O_{F_i}$. Indeed, this is clear if
$\L|_{F_i}=\O_{F_i}$. On the other hand, suppose
$\L|_{F_i}\neq\O_{F_i}$. Let $F_{i,1},\dots,F_{i,\ell_i}$ be the
ordered sequence of components of $F_i$ such that $P_{i,1}\in F_{i,1}$
and $P_{i,2}\in F_{i,\ell_i}$. Since $\L|_{F_i}\neq\O_{F_i}$ there is
a smallest (resp.~largest) integer $j$ such that 
$\deg\L|_{F_{i,j}}\neq 0$; if $\rho_{i,1}\neq 0$ (resp.~$\rho_{i,2}\neq
0$) then $\deg\L|_{F_{i,j}}>0$. However, since $\L$ is
$\psi$-admissible, both maps cannot be simultaneously nonzero.

To summarize, if $\psi_*\L$ is invertible then $\L$ is
$\psi$-admissible. On the other hand, observe that $v'_i$ is
surjective for each $i=1,\dots,e$. Indeed, it follows from
applying $\psi_*$ to the first exact sequence in \eqref{exx} that the
map $\psi_*\L\to\psi_*(\L|_{F_i})$ induced by restriction is surjective,
and thus so is $v'_i$. Thus, if $\L|_{F_i}=\O_{F_i}$, the maps 
$\rho_{i,1}$ and $\rho_{1,2}$ are nonzero, and thus, from 
Diagram \eqref{ddg}, the composition \eqref{cddg} is nonzero for
$j=1,2$, whence $\psi_*\L$ is invertible at $P_i$. So, the converse
holds: If $\L$ is $\psi$-admissible then $\psi_*\L$ is invertible.

Observe now that, for each $i=1,\dots,e$, 
the restriction of the evaluation map $v\:\psi^*\psi_*\L\to\L$ to $F_i$
is a map $v_i\:H^0(P_i,\psi_*\L|_{P_i})\ox\O_{F_i}\to\L|_{F_i}$. Thus, 
if $v$ is surjective then $\L$ is positively $\psi$-admissible, and if 
$v$ is an isomorphism then $\psi_*\L$ is invertible and $\L$ is 
$\psi$-invertible.

Assume from now on that $\L$ is positively $\psi$-admissible. 
Note that each $v_i$ is obtained by composing the base-change map 
$v'_i\:\psi_*\L|_{P_i}\to\psi_*(\L|_{F_i})$ with the evaluation map 
$v''_i\:H^0(F_i,\L|_{F_i})\ox\O_{F_i}\to\L|_{F_i}$. Since 
$\L$ is positively $\psi$-admissible, it follows from Lemma 
\ref{chainh1h0} that 
$$
h^1(F_i,\L|_{F_i})=h^1(F_i,\L|_{F_i}(-Q))=0,
$$
and thus, by the Riemann--Roch Theorem,
$h^0(F_i,\L|_{F_i}(-Q))<h^0(F_i,\L|_{F_i})$ 
for every $Q$ on
the nonsingular locus of $F_i$. So $v''_i$ is surjective. Since the
$v'_i$ was already shown to be surjective, so is $v_i$ for each
$i=1,\dots,e$, whence $v$ is surjective.

Moreover, if $\psi_*\L$ is invertible then $v$ is a surjective map
between 
invertible sheaves, whence an isomorphism.
\end{proof}

\begin{Thm}\label{compadm} Let $X$ be a curve and 
$\psi\:Y\to X$ a semistable modification of $X$. 
Let $\L$ and $\M$ be $\psi$-admissible invertible sheaves on $Y$. 
Assume that $\M\otimes\L^{-1}$ is a twister of $Y$ of the form
$$
\O_Y\left(\sum c_E E \right), \,\,\,  c_E\in\mathbb{Z},
$$
where the sum runs over the components $E$ of $Y$ contracted by
$\psi$. Then $\psi_*\L\simeq\psi_*\M$.
\end{Thm}

\begin{proof} Set $\T:=\M\otimes\L^{-1}$. Let $\R$ be the 
set of smooth, rational curves contained in $Y$ and contracted by
$\psi$. If $\R=\emptyset$, then $\T=\O_Y$ and thus $\L\cong\M$. 
Suppose $\R\ne\emptyset$. Let $\K$ be the set of maximal chains of rational 
curves contained in $\R$. 

{\it Claim:} For every $F\in \K$ and every two components $E_1,E_2\subseteq F$ 
such that $E_1\cap E_2\ne\emptyset$, we have $|c_{E_1}-c_{E_2}|\le 1$. In 
addition, if $E$ is an extreme component of $F$, then $|c_E|\le 1$.

Indeed, let $E_1,\dots E_n$ be the components of $F$, 
ordered in such a way that $\# E_i\cap E_{i+1}=1$ for $i=1,\dots,n-1$. Since  
 $\L$ and $\M$ are admissible, $|\deg_G\T|\le 2$ for every subchain 
$G$ of $F$. Set $c_{E_0}:=c_{E_{n+1}}:=0$. We will reason by contradiction. 
Thus, up to reversing the order of the $E_i$, we may assume 
that $c_{E_i}-c_{E_{i+1}}\ge 2$ 
for some $i\in\{0,\dots,n\}$. Then
$$
c_{E_i}\le c_{E_{i-1}}\le \cdots \le c_{E_1}\le c_{E_0}=0,
$$
because, if  
$c_{E_j}>c_{E_{j-1}}$ for some $j\in\{1,\dots,i\}$, then 
$$
\deg_{E_j\cup\cdots\cup E_i}\T=c_{E_{j-1}}-c_{E_j}+c_{E_{i+1}}-c_{E_i}<-2.
$$
Similarly, $c_{E_{i+1}}\ge c_{E_{i+2}}\ge \cdots \ge c_{E_n}\ge c_{E_{n+1}}=0$. 
But then
$$
0\le c_{E_{i+1}}<c_{E_i}\le 0,
$$
a contradiction that proves the claim.
  
Now, for each $F\in\K$, let $F^\dagger$ be the (possibly empty) union of 
components $E\subseteq F$ such that $c_E=0$. For each connected 
component $G$ of $\overline{F-F^\dagger}$ and irreducible components   
$E_1,E_2\subseteq G$, it follows from the claim that $c_{E_1}\cdot c_{E_2}>0$. 
Let $\K^+$ (resp.~$\K^-$) be the collection of connected components $G$ of  
$\overline{F-F^\dagger}$ for $F\in\K$ such that $c_E>0$ (resp.~$c_E<0$) 
for every irreducible component $E\subseteq G$.

Notice that, again by the claim, 
  \begin{equation}\label{extrcoef}
 c_E=
  \begin{cases}
  \begin{array}{ll}
  \,\,\,\,1 & \text{if $E$ is an extreme component of some $G\in \K^+_F$}\\
  -1 & \text{if $E$ is an extreme component of some $G\in \K^-_F$.}
    \end{array}
  \end{cases}
   \end{equation}
So, being $\L$ and $\M$ admissible, 
 \begin{equation}\label{degF}
 \deg_G \L=-\deg_G\M=
  \begin{cases}
  \begin{array}{ll}
  \,\,\,\,1 & \text{if } G\in \K^+ \\
  -1 & \text{if } G\in \K^-.  
  \end{array}
  \end{cases}
 \end{equation}

Define 
 $$
W^+:=\overline{Y- \cup_{G\in \K^+} G}, 
  \,\,\,\,\,\,\,\,\,\,
  W^-:=\overline{Y-\cup_{G\in \K^-} G},
  \,\,\,\,\,\,\,\,\,\,
  W:=\overline{Y-\cup_{G\in \K^-\cup\K^+} G}.
$$
For each $G\in\K^+\cup \K^-$, let $N_G$ and $N'_G$ denote the points of 
$G\cap \overline{Y-G}$, and put
$$
D^+:=\sum_{G\in\K^+}(N_G+N'_G) \,\,\text{ and }\,\, 
D^-:=\sum_{G\in\K^-}(N_G+N'_G).
$$
We may view $D^+$ and $D^-$ as divisors of $W$. 
Thus, by \eqref{extrcoef}, 
 \begin{equation}\label{rel1}
  \M|_W\simeq \L|_W(D^+-D^-).
 \end{equation}
 
Consider the natural diagram
$$
\begin{CD}
@. @.   @. 0 @. \\
 @. @.  @.   @VVV \\
@. @.   @. \L|_W(-D^-)@. \\
 @. @.  @.   @VVV \\
0@>>>  \underset{G\in \K^+}{\bigoplus} \L|_G(-N_G-N'_G)  
@>>> \L @>>>\L|_{W^+} @>>> 0\\
 @. @.  @.   @VVV \\
@.  @.  @. \underset{G\in \K^-}{\bigoplus} \L|_G   \\
 @. @.  @.   @VVV \\
@.  @.  @. 0  \\
\end{CD}
$$
where the horizontal and vertical sequences are exact. 
By (\ref{degF}) and Lemma \ref{chainh1h0}, and using the 
Riemann--Roch Theorem, 
$$
R^i\psi_* \L|_G(-N_G-N'_G)=H^i(G,\L|_G(-N_G-N'_G))\ox\O_{\psi(G)}=0
$$
for $G\in \K^+$ and $i=0,1$, whereas
$$
\psi_*\L|_G=H^0(G,\L|_G)\ox\O_{\psi(G)}=0 \; \text{ for } G\in \K^-.
$$
Hence, it follows from the above diagram, by considering the 
associated long exact sequences in higher direct images of $\psi$, that
\begin{equation}\label{rel2}
\psi_*\L\simeq (\psi|_W)_*\L|_W(-D^-).
\end{equation}

Consider a second  diagram, similar to the above, but with the 
roles of $\K^+$ and $\K^-$, and thus of $D^+$ and $D^-$, reversed, and 
$\M$ substituted for $\L$. As before, 
$$
R^i\psi_*\M|_G(-N_G-N'_G))=
H^i(G,\M|_G(-N_G-N'_G))\ox\O_{\psi(G)}=0
$$
for $G\in \K^-$ and $i=0,1$, whereas
$$
\psi_*\M|_G\simeq H^0(G,\M|_G)\ox\O_{\psi(G)}=0 \text{ for } G\in \K^+.
$$
Hence, taking the associated long exact sequences, 
\begin{equation}\label{rel3}
\psi_*\M\simeq (\psi|_W)_*\M|_W(-D^+).
 \end{equation} 
 Combining (\ref{rel1}), (\ref{rel2}) and (\ref{rel3}), 
we get $\psi_*\L\simeq \psi_*\M$.
 \end{proof}

\section{Stability}\label{4}

\begin{Thm}\label{famchain2} 
Let $X$ be a curve and $\psi\:Y\to X$ a 
semistable modification of $X$. Let $P$ be a simple point of $Y$ 
not lying on any component contracted by $\psi$. Let $\E$ be a locally
free sheaf on $X$ and $\L$ an invertible sheaf on $Y$. Then 
$\L$ is semistable (resp.~$P$-quasistable, resp.~stable) 
with respect to $\psi^*\E$ if and only if 
$\L$ is $\psi$-admissible (resp.~negatively $\psi$-admissible,
resp.~$\psi$-invertible) 
and $\psi_*\L$ is 
semistable (resp.~$\psi(P)$-quasistable, resp.~stable) 
with respect to $\E$.
\end{Thm}

\begin{proof}
Since $\psi^*\E$ has 
degree 0 on every component of $Y$ contracted by $\psi$, 
and $P$ does not lie on any of these components, 
it follows from the definitions that a semistable (resp.~$P$-quasistable, 
resp.~stable) sheaf has degree $-1$, $0$ or $1$ 
(resp.~$-1$ or $0$, resp. $0$) on every chain of rational curves of
$Y$ contracted by $\psi$.

We may thus assume that $\L$ is $\psi$-admissible. 
Let $W$ be any connected subcurve of $X$. 
Set $W':=\ol{X-W}$ and $\Delta_W:=W\cap W'$. 
Set $\delta:=\#\Delta_W$. Let $V_1:=\ol{Y-\psi^{-1}(W')}$ and 
$V_2:=\ol{Y-\psi^{-1}(W)}$. Let $F_1,\dots,F_r$ be the maximal chains of 
rational curves contained in $\psi^{-1}(\Delta_W)$. Then $0\leq r\leq \delta$. 

{\it Claim:} $(\psi_*\L)_W\cong\psi_*(\L|_Z)$ 
for a certain connected subcurve $Z\subseteq Y$ such that:
\begin{enumerate}
\item $V_1\subseteq Z\subseteq\psi^{-1}(W)$.
\item For each connected subcurve $U\subseteq Y$ 
such that $V_1\subseteq U\subseteq\psi^{-1}(W)$,
$$
\deg(\L|_U)\geq\deg(\L|_Z).
$$
\end{enumerate}
(Notice that Property 1 implies that $P\in Z$ if and only if $\psi(P)\in W$.) 

Indeed, if $W=X$, let $Z:=\psi^{-1}(W)$. Suppose $W\neq X$. Then $\delta>0$.  
Let $M_1,\dots,M_\delta$ be the points of intersection of $V_1$ 
with $V'_1:=\ol{Y-V_1}$ and $N_1,\dots,N_\delta$ those of $V_2$ with 
$V'_2:=\ol{Y-V_2}$. 

Write $F_i=F_{i,1}\cup\dots\cup F_{i,e_i}$, 
where $F_{i,j}\cap F_{i,j+1}\neq\emptyset$ for $j=1,\dots,e_i-1$ and 
$F_{i,1}$ intersects $V_1$. Up to reordering the $M_i$ and $N_i$, 
we may assume that $F_{i,1}$ intersects $V_1$ 
at $M_i$ and $F_{i,e_i}$ intersects $V_2$ at $N_i$ for $i=1,\dots,r$. 
(Thus $M_i=N_i$ for $i=r+1,\dots,\delta$.) Up to 
reordering the $F_i$, we may also assume that there 
are nonnegative integers $u$ and $t$ with $u\leq t$ 
such that 
$$
\deg(\L|_{F_i})=\begin{cases}
1&\text{for $i=1,\dots,u$}\\
0&\text{for $i=u+1,\dots,t$}\\
-1&\text{for $i=t+1,\dots,r$.}
\end{cases}
$$

Up to reordering the $F_i$, we 
may assume there is an integer $b$ with $u\leq b\leq t$ such that, 
for each $i=u+1,\dots,t$, we have that $i>b$ if and only if 
$\deg(\L|_{F_{i,j}})=0$ for every $j$ or the largest integer $j$ such that 
$\deg(\L|_{F_{i,j}})\neq 0$ is such that $\deg(\L|_{F_{i,j}})=-1$. Set 
$G_i:=F_i$ for $i=b+1,\dots,r$. For each $i=u+1,\dots,b$, let 
$G_i:=F_{i,1}\cup\dots\cup F_{i,j-1}$, where $j$ is the largest integer such 
that $\deg(\L|_{F_{i,j}})=1$, let $\wh G_i:=\ol{F_i-G_i}$ and denote by 
$B_i$ the point of intersection of $G_i$ and $\wh G_i$. (Notice 
that $1< j\leq e_i$.) Let $B_i:=M_i$ and $\wh G_i:=F_i$ for $i=1,\dots,u$, and 
$B_i:=N_i$ for $i=b+1,\dots,\delta$.  

For $i=u+1,\dots,r$, since the degree of $\L|_{G_i}(B_i)$ on each subchain of 
$G_i$ is at most 1, it follows from Lemma~\ref{chainh1h0} that
\begin{equation}\label{MN3}
h^0(G_i,\L|_{G_i}(-M_i))=0\quad\text{for $i=u+1,\dots,r$}.
\end{equation}
Furthermore, for $i=1,\dots,b$, the total degree of $\L|_{\wh G_i}$ is 1; thus, 
by Lemma~\ref{chainh1h0} and the Riemann--Roch Theorem,
\begin{equation}\label{MN4}
h^1(\wh G_i,\L|_{\wh G_i}(-B_i-N_i))=0\quad\text{for $i=1,\dots,b$}.
\end{equation}

Set
$$
Z:=V_1\cup G_{u+1}\cup\dots\cup G_r
$$
and $Z':=\ol{Y-Z}$. Put $\Delta_Z:=Z\cap Z'$. Notice that 
$\Delta_Z=\{B_1,\dots,B_\delta\}$. Also, notice that $Z$ is connected, and 
$$
\deg(\L|_U)\geq\deg(\L|_Z)=\deg(\L|_{V_1})-(b-u)-(r-t)
$$
for each connected subcurve $U\subseteq Y$ such that 
$V_1\subseteq U\subseteq\psi^{-1}(W)$.

We have three natural exact sequences:
\begin{equation}\label{ex1}
0\to\L|_{Z'}\Big(-\sum_{i=1}^\delta B_i\Big)\to\L\to\L|_Z\to 0,
\end{equation}
\begin{equation}\label{ex2}
0\to\bigoplus_{i=1}^b\L|_{\wh G_i}(-B_i-N_i)\to
\L|_{Z'}\Big(-\sum_{i=1}^\delta B_i\Big)
\to\L|_{V_2}\Big(-\sum_{i=b+1}^\delta B_i\Big)\to 0,
\end{equation}
\begin{equation}\label{ex3}
0\to\bigoplus_{i=u+1}^r\L|_{G_i}(-M_i)\to\L|_Z\to\L|_{V_1}\to0.
\end{equation}
Since $\L$ is $\psi$-admissible, so are 
$\L|_{V_1}$ with respect to $\psi|_{V_1}\:V_1\to W$ and $\L|_{V_2}$ with 
respect to $\psi|_{V_2}\:V_2\to W'$. Then $\psi_*(\L|_{V_1})$ is a 
torsion-free, rank-1 sheaf on $W$ and 
$R^1\psi_*(\L|_{V_2}(-\sum B_i))=0$ by Theorem \ref{famchain}. 

Since $R^1\psi_*(\L|_{V_2}(-\sum B_i))=0$, from 
\eqref{MN4} and the long exact sequence of higher direct images under 
$\psi$ of \eqref{ex2} and \eqref{ex1} we get that 
$R^1\psi_*(\L|_{Z'}(-\sum B_i))=0$ 
and the natural map $\psi_*\L\to\psi_*(\L|_Z)$ is surjective. Also, it 
follows from \eqref{MN3} and the long exact sequence of 
higher direct images under $\psi$ of \eqref{ex3} that the natural map 
$\psi_*(\L|_Z)\to\psi_*(\L|_{V_1})$ is injective. Thus, since 
$\psi_*(\L|_{V_1})$ is a torsion-free, rank-1 sheaf on $W$, so is 
$\psi_*(\L|_Z)$. And, since $\psi_*\L\to\psi_*(\L|_Z)$ is surjective, we 
get an isomorphism $(\psi_*\L)_W\cong\psi_*(\L|_Z)$, 
finishing the proof of the claim.

To prove the ``only if'' part, let $W$ be any connected 
subcurve of $X$. Let $Z$ be as in the claim. Since $\L$ is admissible 
with respect to $\psi$, Theorem \ref{famchain} 
yields $R^1\psi_*\L=0$, and hence 
$R^1\psi_*(\L|_Z)=0$ from the long exact sequence of higher direct images 
under $\psi$ of \eqref{ex1}. Thus, by the claim and the 
projection formula,  
\begin{equation}\label{ZW}
\chi((\psi_*\L)_W\ox\E|_W)=\chi(\psi_*(\L|_Z)\ox\E|_W)=
\chi(\L|_Z\ox(\psi^*\E)|_Z).
\end{equation}
If $\L$ is semistable (resp.~$P$-quasistable, resp.~stable) then 
$\chi(\L|_Z\ox(\psi^*\E)|_Z)\geq 0$ (resp.~with equality only if $Z=Y$ or 
$Z\not\ni P$, resp.~with equality only if $Z=Y$). Now, if 
$Z=Y$ then $W=X$. Also, $P\in Z$ if and only if $\psi(P)\in W$. So 
\eqref{ZW} yields $\chi((\psi_*\L)_W\ox\E|_W)\geq 0$ 
(resp.~with equality only if $W=X$ or 
$W\not\ni\psi(P)$, resp.~with equality only if $W=X$).

As for the ``if'' part, let $U$ be a connected subcurve of $Y$. 
If $U$ is a union of components of $Y$ contracted by $\psi$, then
$U$ is a chain of rational curves of $Y$ collapsing to a node of $X$, 
and hence $\L|_U$ has degree at least $-1$ 
(exactly $0$ if $\L$ is $\psi$-invertible). Thus 
$$
\chi(\L|_U\ox\psi^*\E|_U)=\text{rk}(\E)\chi(\L|_U)\geq 0,
$$
with equality only if $\L$ is not $\psi$-invertible.

Suppose now that $U$ contains a component of $Y$ not contracted by
$\psi$. Then $W:=\psi(U)$ is a connected subcurve of $X$. 
Let $\wh U$ be the smallest subcurve of $Y$ containing $U$ and 
$\ol{Y-\psi^{-1}(W')}$, where $W':=\ol{X-W}$. Then $\wh U$ is connected 
and contained in $\psi^{-1}(W)$. Furthermore, $\chi(\O_U)-\chi(\O_{\wh U})$ 
is the number of connected components of $\ol{\wh U-U}$. Thus 
\begin{equation}\label{UU}
\deg(\L|_U)+\chi(\O_U)\geq\deg(\L|_{\wh U})+\chi(\O_{\wh U}),
\end{equation}
with equality only if $\L$ has degree 1 on every connected component of 
$\ol{\wh U-U}$. Let $Z$ be as in the claim. Notice that 
$\chi(\O_{\wh U})=\chi(\O_Z)$. Since 
$\deg(\L|_{\wh U})\geq\deg(\L|_Z)$ by the claim,
using \eqref{ZW} and \eqref{UU} we get
\begin{align*}
\chi(\L|_U\ox\psi^*\E|_U)&=\text{rk}(\E)(\deg(\L|_U)+\chi(\O_U))
+\deg(\psi^*\E|_U)\\
&\geq\text{rk}(\E)(\deg(\L|_{\wh U})+\chi(\O_{\wh U}))
+\deg(\psi^*\E|_{\wh U})\\
&=\text{rk}(\E)(\deg(\L|_{\wh U})+\chi(\O_Z))
+\deg(\psi^*\E|_Z)\\
&\geq\text{rk}(\E)(\deg(\L|_Z)+\chi(\O_Z))
+\deg(\psi^*\E|_Z)\\
&=\chi(\L|_Z\ox(\psi^*\E)|_Z)\\
&=\chi((\psi_*\L)_W\ox\E|_W).
\end{align*}
Assume that $\psi_*\L$ is 
semistable (resp.~$\psi(P)$-quasistable, resp.~stable) 
with respect to $\E$. Then $\chi((\psi_*\L)_W\ox\E|_W)\geq 0$ 
(resp.~with equality only if $W=X$ or 
$W\not\ni\psi(P)$, resp.~with equality only if $W=X$). So 
$\chi(\L|_U\ox\psi^*\E|_U)\geq 0$. Suppose $\chi(\L|_U\ox\psi^*\E|_U)=0$. 
Then $\chi((\psi_*\L)_W\ox\E|_W)=0$ and equality holds in \eqref{UU}. 
If $W\not\ni\psi(P)$ then $U\not\ni P$. Suppose $W=X$. Then 
$\wh U=Y$. If $U\neq Y$ then $\L$ has degree 1 on each connected component of 
$\ol{Y-U}$, and thus $\L$ is not negatively $\psi$-admissible. 
\end{proof}

\section{Sheaves on quasistable curves}\label{5}

If $X$ is a semistable curve, a stable 
curve $\check X$ may be obtained from $X$ by contracting all exceptional 
components. We say that $\check X$ is the \emph{stable model} of $X$. 

Let $f\col Y\ra S$ be a family of semistable curves. We call 
the pair $(\check f,\psi)$, consisting of a family
of stable curves $\check f\:X\to S$ and a $S$-map
$\psi\:Y\to X$, a \emph{stable model} of $f$ if
$\psi$ is a semistable 
modification
of $\check f$. So, for every geometric point $s$ of
$S$ the induced map $\psi_s\:Y_s\to X_s$
is the map contracting
all exceptional components of $Y_s$. We will also call 
$\check f$ the stable model of $f$ and $\psi$ the \emph{contraction map}.

Stable models always exist, and are unique up
to unique isomorphism by the following proposition.

\begin{Prop}\label{exstmod}
Let $f\:Y\to S$ be a family of semistable curves. The following 
statements hold:
\begin{enumerate}
\item The family $f$ has a stable model.
\item If $\check f\:X\to S$ and $\check f'\:X'\to S$ are stable models of $f$,
with contraction maps $\psi\:Y\to X$ and $\psi'\:Y\to X'$, then there is a unique
isomorphism $u\:X'\to X$ such that $\check f'=\check fu$ and $\psi=u\psi'$.
\item For each stable model $\check f$ with contraction map $\psi$, the 
comorphism $\O_X\to\psi_*\O_Y$ is an isomorphism, $R^1\psi_*\L=0$, and
the pullback of the relative dualizing sheaf of $\check f$ under $\psi$ is the 
relative dualizing sheaf of $f$.
\end{enumerate} 
\end{Prop}

\begin{proof} We will prove Statement 3 first. 
So, let $\check f\:X\to S$ be a stable 
model of $f$ with contraction map $\psi\:Y\to X$. Then
$R^1\psi_*\O_Y=0$ by Theorem \ref{famchain}. Furthermore, 
$\psi_*\O_Y$ is invertible and the evaluation map
$v\:\psi^*\psi_*\L\to\L$ is an isomorphism. If
$\psi^\#\:\O_X\to\psi_*\O_Y$ is 
the comorphism, since $v\psi^*(\psi^\#)$ is a natural isomorphism, it follows
that $\psi^*(\psi^\#)$ is an isomorphism, and thus that $\psi^\#$ is
surjective. Since $\psi^\#$ is a surjection between invertible
sheaves, it is an isomorphism.
 
Let $\check\omega$ be the relative dualizing sheaf of $\check f$. Then 
\begin{equation}\label{rd1}
\begin{aligned}
R^1f_*(\psi^*\check\omega)=&R^1\check f_*\psi_*(\psi^*\check\omega)=
R^1\check f_*(\check\omega\ox \psi_*\O_Y)\\
=&R^1\check f_*(\check\omega)=\O_S.
\end{aligned}
\end{equation}
Indeed, the fourth equality in \eqref{rd1} is given by the trace map, 
an isomorphism because the fibers of $\check f$ are connected. The
third equality follows from $\O_X=\psi_*\O_Y$.
The projection formula, which holds because $\check\omega$ is invertible, yields 
the second equality. Finally, the first equality holds because of the degeneration 
of the spectral sequence associated to the composition $\check f\psi$, since
$$
R^1\psi_*(\psi^*\check\omega)=\check\omega\ox R^1\psi_*\O_Y=0.
$$ 

Let $\omega$ be the relative dualizing sheaf of $f$. By \cite{reldual}, Thm.~21, p.~55, 
for each coherent sheaf $\N$ on $S$,
there is a functorial (on $\N$) isomorphism
\begin{equation}\label{rd}
f_*Hom(\psi^*\check\omega,\omega\ox f^*\N)\to Hom(R^1f_*(\psi^*\check\omega),\N).
\end{equation}
Putting the isomorphisms \eqref{rd1} and \eqref{rd} 
together, we get a functorial (on $\N$) isomorphism
\begin{equation}\label{functN}
f_*Hom(\psi^*\check\omega,\omega\ox f^*\N)\to\N.
\end{equation}
In particular, replacing $\N$ by $\O_S$, we get
a natural map $h\:\psi^*\check\omega\to \omega$, corresponding
to the constant function $1_S$. This map is fiberwise
(over $S$) nonzero, a fact that can be shown by
replacing $\N$ by skyscraper sheaves and using
the functoriality of \eqref{functN}.

Since both $\check\omega$ and $\omega$ are invertible, we need only show that 
$h$ is surjective, and thus we may assume that
$S$ is the spectrum of an algebraically closed field.
Now, $\omega$ and $\psi^*\check\omega$ restrict to isomorphic sheaves on
each component $Z$ of $Y$. In fact, it follows from adjunction that
$$
\omega|_Z\cong\L\ox\O_Z\left(\sum_{P\in Z\cap Z'}P\right)\cong
 \psi^*\check\omega|_Z,
$$
where $\L$ is the dualizing sheaf of $Z$. In particular, 
$\omega$ and $\psi^*\check\omega$ have the same multidegree.
Since $h$ is nonzero,
it follows that $h$ is an isomorphism.

We will now prove Statement 1. Let $\omega$ be the relative dualizing sheaf
of $f$ and consider the $S$-scheme:
$$
X:=\text{Proj}_S\big(\O_S\oplus f_*\omega\oplus
f_*(\omega^{\ox 2})\oplus\cdots\big).
$$
Let $\check f\:X\to S$ denote the structure map. 

For each geometric point $s$ of 
$S$, by adjunction, $\omega_s$ has positive degree on each nonexceptional component of 
$Y_s$, and thus, by duality, 
$$
H^1(Y_s,\omega_s^{\ox n})=H^0(Y_s,\omega_s^{\ox 1-n})^*=0\quad\text{for each $n\geq 2$.}
$$
It follows that the direct image $f_*(\omega^{\ox n})$ is locally free, with formation 
commuting with base change, for each $n\geq 2$. Also, $f_*\omega$ is locally free, 
with formation commuting with base change, 
because $R^1f_*\omega\cong\O_S$, the trace map being an isomorphism.
So, $\check f$ is flat, and its formation commutes with base change, so
\begin{equation}\label{Xsf}
X_s=\text{Proj}\big(H^0(Y_s,\O_{Y_s})\oplus H^0(Y_s,\omega_s)\oplus
H^0(Y_s,\omega_s^{\ox 2})\oplus\cdots\big)
\end{equation}
for each geometric point $s$ of $S$.

By \cite{Cat}, Thm.~A, p.~68, the sheaf $\omega_s^{\ox n}$ is globally generated 
for each integer $n\geq 2$ and each geometric points $s$ of $S$. 
Thus, the natural maps $f^*f_*(\omega^{\ox n})\to\omega^{\ox n}$ are surjections for
$n\geq 2$, and hence induce a globally defined $S$-map $\psi\: Y\to X$. 

We need only show now that, for each geometric point $s$ of $S$, the scheme $X_s$ is a stable model of $Y_s$ and 
$\psi_s$ is a contraction map. Indeed, let $Z$ be a stable model of $Y_s$, 
and let $b\:Y_s\to Z$ be a contraction map. Let $\L$ be the dualizing sheaf of 
$Z$. Then $b^*\L\cong\omega_s$ by Statement 3. Since $b_*\O_{Y_s}=\O_Z$, it 
follows that
\begin{equation}\label{LYZ}
H^0(Z,\L^{\ox n})=H^0(Y_s,\omega_s^{\ox n})\quad\text{for each integer $n>0$.}
\end{equation}
On the other hand, since $Z$ is stable, $\L$ is 
ample, and thus
$$
Z=\text{Proj}\big(H^0(Z,\O_Z)\oplus H^0(Z,\L)\oplus
H^0(Z,\L^{\ox 2})\oplus\cdots\big).
$$
It follows now from \eqref{Xsf} and \eqref{LYZ} that there is an isomorphism 
$u\:Z\to X_s$ such that $\psi_s=ub$.
\end{proof}

If $X$ is a scheme and $\Fcal$ is a
coherent sheaf on $X$, let
$$
Sym(\Fcal)=\bigoplus_{n\geq 0}Sym^n(\Fcal)\quad\text{and}\quad
\IP_X(\Fcal):=\text{\rm Proj}(Sym(\Fcal)),
$$
where $Sym^n(\Fcal)$ is the $n$th symmetric product of $\Fcal$, for 
each integer $n\geq 0$.

\begin{Prop}
\label{PXI}
Let $X$ be a curve and $\I$ a torsion-free, rank-1
sheaf on $X$. Set $Y:=\IP_X(\I)$, and let $\psi\:Y\to X$ be
the structure map. Then $\psi$ is a small semistable modification of
$X$. The exceptional components of $Y$ contracted by $\psi$ are the
fibers of $\psi$ over the points of $X$ where $\I$ is not
invertible. In particular, if $X$ is stable, then $Y$ is quasistable
with stable model $X$ and contraction map $\psi$.
\end{Prop}

\begin{proof} Wherever $\I$ is invertible,
$\psi$ is an isomorphism. So, let us analyze $\psi$ on a
neighborhood of a node $P$ of $X$ where $\I$ fails to
be invertible. In fact, consider the
base change of $\psi$ to the
spectrum of the completion $\wh\O_{X,P}$. Since $P$ is a node, where $\I$ fails to 
be invertible, $\wh\I_P\cong\mathfrak m_P$, where $\mathfrak m_P$
is the maximal ideal of $\wh\O_{X,P}$. Also, since $P$ is a
node,
$$
\wh\O_{X,P}\cong\frac{K[[u,v]]}{(uv)},
$$
where $K$ is the base field of $X$. Now, under the above identification,
$$
\mathfrak m_P\cong\frac{\wh\O_{X,P}\oplus\wh\O_{X,P}}
{v\wh\O_{X,P}\oplus u\wh\O_{X,P}}
$$
as  an $\wh\O_{X,P}$-module. 
So, locally analytically, $Y$ is the subscheme of
$\text{\bf A}^2_K\times\IP^1_K$ defined by the equations
$uv=sv=tu=0$, where $u$ and $v$ are the coordinates
of $\text{\bf A}^2_K$ and $s$ and $t$ are
homogeneous coordinates of $\IP^1_K$. Also, $\psi$
is the restriction to $Y$ of the projection
$\text{\bf A}^2_K\times\IP^1_K\to\text{\bf A}^2_K$ onto the first factor. Then 
$Y$ is the union of three lines, the
projective line given by $u=v=0$, and the affine lines
given by $u=s=0$ and $v=t=0$, the latter two not meeting each other, but intersecting the former 
transversally.

As the above reasoning applies to any node $P$ of $X$ where $\I$ fails to be invertible, it follows 
that the singularities of $Y$ are nodes, that $Y$ is a curve, and that $\psi^{-1}(P)$ is a smooth, rational 
component of $Y$ with $k_{\psi^{-1}(P)}=2$ for any such $P$. 
\end{proof}

\begin{Lem}
\label{eskl}
{\rm (E--Kleiman)}
Let $p\:X\to S$ be a flat map and $\Fcal$ a $S$-flat coherent sheaf on 
$X$. Assume
$\Fcal$ is invertible at each associated point of $X$, and
is everywhere locally generated by two sections.
Set $W:=\IP_X(\Fcal)$, and let $w\: W\to X$ be the
structure map. Then $W$ is $S$-flat and
Serre's graded $\O_X$-algebra homomorphism
$$
Sym(\Fcal)\longrightarrow\bigoplus_{n\geq 0}w_*\O_W(n)
$$
is an isomorphism.
\end{Lem}

\begin{proof} We refer to the proof
of \cite{ek}, Lemma 3.1, p.~491 and its notation. To complement the proof, 
we need only observe 
that $W$ is $S$-flat. First, notice
that $\N$ is $S$-flat, because
of the first exact sequence
in the proof. Second, notice that $V$ is $S$-flat,
being a projective bundle over $X$. The structure map is
denoted $v\:V\to X$. Since $\N$ is flat, and $v$ is a
projective bundle map, it follows from the third exact
sequence in the proof that $W$ is a subscheme of $V$ with
a $S$-flat sheaf of ideals. Now, the formation of this third exact
sequence commutes with base change. So $W$ is $S$-flat.
\end{proof}

\begin{Prop}\label{L2I} Let $f\:Y\to S$ be a family of
quasistable curves. Let $\L$ be an invertible sheaf
on $Y$ of degree $d$ on $Y/S$
such that $\deg_E(\L)=1$ for every
exceptional component $E$ of every geometric fiber of $Y/S$.
Let $\check f\:X\to S$ be a stable model of $f$ and
$\psi\:Y\to X$ the contraction map. Let $\I:=\psi_*\L$.
Then the following statements hold:
\begin{enumerate}
\item The direct image $\psi_*\L$ is a torsion-free, rank-1 sheaf on
 $X/S$ of relative degree $d$, whose formation commutes with base change.
\item For each geometric point
$s\in S$ and each node $P$ of $X_s$, the sheaf $\I_s$ is
invertible at $P$ if and only if $\psi$ is an isomorphism
over a neighborhood of $P$.
\item The evaluation map $e\:\psi^*\I\to\L$ is surjective.
\item There is an isomorphism
$u\:Y\to\IP_X(\I)$ over $X$ such that
$u^*\O(1)\cong\L$.
\end{enumerate}
\end{Prop}

\begin{proof} Statement~1 follows readily from Theorem
  \ref{famchain}, as well as Statement~3. It follows from Statement~3 
that $e$ defines a 
$X$-map $u\:Y\to\IP_X(\I)$ such that $u^*\O(1)\cong\L$.
Then, to prove Statement~4, since
both $Y$ and $\IP_X(\I)$ are $S$-flat, the latter by
Lemma~\ref{eskl}, and the formation of $\I$ commutes with base change
by Statement~1, we need only check that
$u_s$ is an isomorphism for every geometric point $s$
of $S$.

So, for the remainder of the proof, we may now
assume $S$ is the
spectrum of an algebraically closed field. 

The contraction map $\psi$ factors as the composition of several maps,
each contracting a single exceptional component. Thus, to prove
Statement~2 we may assume that $\psi$ contracts a single
component. Then Statement~2 follows from Theorem \ref{famchain} as
well.

As for Statement~4, first observe that $\IP_X(\I)$ is a
quasistable curve isomorphic to $Y$, by 
Proposition \ref{PXI} and Statement 2. 
So, since $u$ is an $X$-morphism, to check that $u$ is an
isomorphism we need only check that, for each exceptional
component $F\subset Y$, the restriction $u|_F$ is an
isomorphism onto the corresponding exceptional component
of $\IP(\I)$. But this is so, because, letting
$R\in X$ denote the point below $F$, the restriction
$u|_F$ is the map to $\IP(\I|_R)$ given by the surjection $e|_F$. 
So, $u|_F$ is an
isomorphism because $e|_F$ is the evaluation map of the degree-1 sheaf $\L|_F$.
\end{proof}

\begin{Prop}\label{I2L}
Let $f\:X\to S$
be a family of curves.
Let $\I$ be a torsion-free, rank-1 sheaf
of degree $d$ on $X/S$. Let $Y:=\IP_X(\I)$,
with structure map $\psi\:Y\to X$,
and let $\L$ denote the tautological invertible
sheaf on $Y$. Then $\psi$ is a small semistable modification of
$X/S$. In particular, if $X/S$ is a family of stable curves, then 
$Y/S$ is a family of quasistable
curves, $X/S$ is its stable model, and $\psi$ is the
contraction map. Furthermore, $\L$ has degree $d$ on $Y/S$, 
the degree of $\L$ on every 
exceptional component contracted by $\psi$ 
of every geometric fiber of $Y/S$ is
$1$, and $\I=\psi_*\L$.
\end{Prop}

\begin{proof} We apply Lemma \ref{eskl} for $\Fcal:=\I$.
The hypotheses are verified because the associated points
of $X$ are generic points of certain fibers of $f$, where $\I$
is invertible, and $\I$ is everywhere locally generated
by two sections, since $X/S$ is a family of nodal curves.
So $Y$ is $S$-flat.

It follows from Lemma \ref{eskl} as well that
$\I=\psi_*\L$. Since 
the formation of $\IP_X(\I)$ commutes with base change,
it follows from Proposition \ref{PXI} that $\psi$ is a semistable 
modification of $X/S$. 

By Proposition \ref{PXI}, the exceptional components contracted by 
$\psi$ of
the geometric fibers of $Y/S$ are the fibers of
$\IP_X(\I)$ over the nodes of the geometric fibers of
$X/S$ where $\I$ is not invertible. Since
$\L$ is the tautological sheaf of $\IP_X(\I)$, its
restriction to a fiber over $X$ is also tautological. So
$\L$ has degree 1 on every exceptional component contracted by 
$\psi$ of
every geometric fiber of $Y/S$.
Finally, that $\L$ has relative degree $d$ over $S$
follows now from Statement 1 of Proposition \ref{L2I}.
\end{proof}

\section{Functorial isomorphisms}\label{fi}

 Let $\P_{d,g}$ be the
contravariant
functor from the category of schemes to that of sets defined
in the following way: For each scheme $S$, let
$\P_{d,g}(S)$ be the set of equivalence classes of
pairs $(f,\L)$, where $f\:Y\to S$ is a family of quasistable curves of
genus $g$ over $S$, and $\L$ is an invertible sheaf on $Y$ of relative degree
$d$ over $S$ whose degree on every exceptional component
of every geometric fiber of $Y/S$ is 1.
Two such pairs, $(f\:Y\to S,\L)$ and
$(f'\:Y'\to S,\L')$, are said to be equivalent if there
are an $S$-isomorphism $u\:Y\to Y'$ and an invertible sheaf
$\N$ on $S$ such that $u^*\L'\cong\L\ox f^*\N$. We leave
it to the reader to define the functor on maps.

On the other hand, let $\Jcal_{d,g}$ be the contravariant
functor from the category of schemes to that of sets defined
in the following way: For each scheme $S$, let
$\Jcal_{d,g}(S)$ be the set of equivalence classes of
pairs $(f,\L)$, where  $f\:X\to S$ is a family of stable curves of
genus $g$ over $S$, and $\I$ is a torsion-free, rank-1 sheaf on $X/S$
of relative degree $d$. Two such pairs,
$(f\:X\to S,\I)$ and
$(f'\:X'\to S,\I')$, are said to be equivalent if there
are an $S$-isomorphism $u\:X\to X'$ and an invertible sheaf
$\N$ on $S$ such that $u^*\I'\cong\I\ox f^*\N$.
Again, we leave it to the reader to define the
functor on maps.

Finally, let $\ol\M_g$ be the usual moduli functor of
stable curves of genus $g$. There are natural ``forgetful''
maps of functors
$\P_{d,g}\to\ol\M_g$, defined
by taking a pair $(f\:Y\to S,\L)$ to the stable model
$X/S$ of $Y/S$, and
$\Jcal_{d,g}\to\ol\M_g$, defined by
taking a pair $(f\:X\to S,\I)$ to $X/S$. The former forgetful 
map is well-defined by Proposition \ref{exstmod}.

\begin{Thm}\label{isof}
There is a natural isomorphism of functors
$$
\Phi\:\P_{d,g}\lra\Jcal_{d,g}
$$
over $\ol\M_g$. The isomorphism $\Phi$ takes a pair
$(f\:Y\to S,\L)$ of a family of quasistable
curves $f$ and an invertible sheaf $\L$ on $Y$
to $(X\to S,\psi_*\L)$, where $X/S$ is the
stable model of $Y/S$ and $\psi\:Y\to X$ is the contraction
map. Its inverse takes a pair 
$(f\:X\to S,\I)$ of a family of stable curves $f$ and a 
torsion-free, rank-1 sheaf $\I$
on $X/S$ to $(\IP_X(\I)\to S,\O(1))$. 
\end{Thm}

\begin{proof} Just combine Propositions \ref{L2I} and
\ref{I2L}.
\end{proof}

Let $g$ and $d$ be integers with $g\geq 2$. Let $Y$ be a curve of genus $g$, and $\omega$ 
a dualizing sheaf of $Y$. The degree-$d$ \emph{canonical polarization}
of $Y$ is the sheaf
\[
\E_d:=\O_Y^{\oplus 2g-3}\oplus\omega^{\otimes g-1-d}.
\]
Let $\I$ be a torsion-free, 
rank-1 sheaf on $Y$ of degree $d$. We say that $\I$ is 
\emph{semistable} (resp.~\emph{stable}) if $\I$ is semistable
(resp.~stable) with respect to $\E_d$. 

Since $\chi(\I)=d+1-g$, and thus
$$
\chi(\I_Z\otimes\E_d|_Z)=(2g-2)\chi(\I_Z)-\chi(\I)\deg_Z(\omega)
$$
for every subcurve $Z\subseteq Y$, it follows that $\I$ is 
semistable (resp.~stable) if and only if
\begin{equation}\label{ssI1}
\chi(\I_Z)\geq\frac{\deg_Z(\omega)}{2g-2}\chi(\I)
\end{equation}
for every subcurve $Z\subseteq Y$ (resp.~with equality only if $Z=Y$). 

If $Y$ is stable, the above 
condition is the same as Seshadri's in \cite{Sesh}, Part 7, Def.~9, p.~153, 
when the 
polarization chosen (in Seshadri's sense) is the so-called canonical: 
If $Y_1,\dots,Y_p$ denote 
the components of $Y$, the 
\emph{canonical polarization} is the 
$p$-tuple $\mathfrak a:=(a_1,\dots,a_p)$ where 
$$
a_i:=\frac{\deg_{Y_i}(\omega)}{2g-2}.
$$
That $\mathfrak a$ is indeed a polarization in 
Seshadri's sense follows from the ampleness of $\omega$, by the stability 
of $Y$. That the above notion of (semi)stability is Seshadri's follows 
from the fact that the nonzero torsion-free quotients of $\I$ 
are the sheaves $\I_Z$ for subcurves $Z$ of $Y$. 

On the other hand, 
$\chi(\I_Z)=\deg_Z(\I)+\chi(\O_Z)$ for each subcurve $Z$ of $Y$. 
Also, it follows from adjunction 
and duality that
$$
\deg_Z(\omega)=\deg(\Fcal)+k_Z=\chi(\Fcal)-\chi(\O_Z)+k_Z=
-2\chi(\O_Z)+k_Z,
$$
where $\F$ is the dualizing sheaf of $Z$. 
Thus, \eqref{ssI1} holds for each proper subcurve $Z\subset Y$ 
if and only if
\begin{equation}\label{ssI2}
\deg_Z(\I)\geq d\left(\frac{\deg_Z(\omega)}{2g-2}\right)-
\frac{k_Z}{2},
\end{equation}
with equality if and only if equality holds in \eqref{ssI1}.

Let $X/S$ be a family of stable curves. A torsion-free, rank-1 
sheaf $\I$ on $X/S$ is said to be \emph{semistable}
(resp.~\emph{stable}) if 
$\I_s$ is semistable (resp.~stable) for each geometric point
$s$ of $S$. Let $\Jcal^{ss}_{d,g}$ (resp.~$\Jcal^s_{d,g}$) 
denote the subfunctor of $\Jcal_{d,g}$ parameterizing the pairs 
$(X/S,\I)$ with $\I$ semistable (resp.~stable) on $X/S$.

According to \cite{CCaCo}, Def.~5.1.1, p.~3756, if $Y$ is quasistable, 
a degree-$d$ invertible sheaf $\L$ on $Y$ is called \emph{balanced} if 
$\deg_E(\L)=1$ for each exceptional 
component $E$ of $Y$ and the ``Basic Inequality'' holds,
\begin{equation}\label{BI}
\Bigg|\deg_Z(\L)-d\left(\frac{\deg_Z(\omega)}{2g-2}\right)\Bigg|\leq
\frac{k_Z}{2},
\end{equation}
for every proper subcurve $Z\subset Y$. Furthermore, 
$\L$ is called \emph{stably balanced} 
if $\L$ is balanced and equality holds in 
\begin{equation}\label{IBI}
\deg_Z(\L)\geq d\left(\frac{\deg_Z(\omega)}{2g-2}\right)
-\frac{k_Z}{2}
\end{equation} 
only if $Z'$ is a union of exceptional components of $Y$. 

Notice that \eqref{BI} for every proper subcurve $Z\subset Y$ is
equivalent to \eqref{IBI} for every proper subcurve $Z\subset Y$,
which is in turn equivalent to 
\begin{equation}\label{uBI}
\deg_Z(\L)\leq d\left(\frac{\deg_Z(\omega)}{2g-2}\right)
+\frac{k_Z}{2}
\end{equation}
for every proper subcurve $Z\subseteq Y$. In addition, if
$\deg_E(\L)=1$ 
for each exceptional component $E$ of
$Y$, then equality holds in \eqref{IBI}  (resp.~\eqref{uBI}) if $Z'$
(resp.~$Z$) is a union of exceptional 
components of $Y$. So, in a formulation analogous to that of
semistability and 
stability, $\L$ is balanced (resp.~stably balanced) if $\deg_E(\L)=1$ for each 
exceptional component $E$ of $Y$ and \eqref{IBI} holds 
for every proper subcurve $Z\subset Y$ (resp.~with equality only if
$Z'$ is a union of exceptional components of $Y$).

Let $Y/S$ be a family of quasistable curves. An invertible
sheaf $\L$ on $Y$ is said to be balanced (resp.~stably balanced) 
on $Y/S$ if
$\L_s$ is balanced (resp.~stably balanced) 
on $Y_s$ for each geometric point
$s$ of $S$. Let $\P^{b}_{d,g}$ (resp.~$\P^{sb}_{d,g}$) 
denote the subfunctor of $\P_{d,g}$ parameterizing the pairs 
$(Y/S,\L)$ with $\L$ balanced (resp.~stably balanced) on $Y/S$.

\begin{Prop}\label{biss}
Let $Y$ be a quasistable curve. Let $X$ be 
its stable model and
$\psi\:Y\to X$ the contraction map.
Let $\L$ be an invertible sheaf
on $Y$ such that $\deg_E(\L)=1$ for every
exceptional component $E\subset Y$. Then
$\L$ is balanced (resp.~stably balanced) if and only if $\psi_*\L$ is 
semistable (resp.~stable).
\end{Prop}

\begin{proof} Let $d$ be the degree of $\L$ and $\E_d$ the degree-$d$ 
canonical polarization on $X$. Let $\check{\omega}$ be a 
dualizing sheaf of $X$. It follows from 
Proposition~\ref{exstmod} that $\omega:=\psi^*\check{\omega}$ 
is a dualizing sheaf 
of $Y$. Thus $\psi^*\E_d$ is the degree-$d$ canonical polarization of
$Y$. 

Since $\L$ is $\psi$-admissible, it 
follows from Theorem~\ref{famchain} that $\psi_*\L$ is torsion-free,
rank-1 and of degree $d$. Define the invertible sheaf 
\begin{equation}\label{bal-semist}
\I:=\L\otimes\O_Y\left(\textstyle\sum E\right)
\end{equation}
on $Y$, where $E$ runs over the set of components of $Y$ contracted by 
$\psi$. Then $\I$ is negatively $\psi$-admissible. Furthermore, 
$\psi_*\I=\psi_*\L$ by Theorem~\ref{compadm}. We claim first that $\L$ 
  is balanced if and only if $\I$ is semistable. Furthermore, let 
$X_1,\dots,X_p$ be all the components of $X$, and $Y_1,\dots,Y_p$
those of $Y$ such that $\psi(Y_i)=X_i$ for $i=1,\dots,p$. We have that
$\psi_*\I$ is stable if and only if $\psi_*\I$ is $X_i$-quasistable 
with respect to
$\E_d$ for every $i=1,\dots,p$. We claim as well that $\L$ is stably 
balanced if and only if $\I$ is $Y_i$-quasistable with respect to
$\psi^*\E_d$ for every $i=1,\dots,p$. Once the claims are proved, 
an application of 
Theorem~\ref{famchain2} finishes the proof of the proposition.

Let $Z$ be a proper subcurve of $Y$. If $Z$ is a union of exceptional 
components of $Y$, then 
\[
\deg_Z(\L)=-\deg_Z(\I)=k_Z/2,
\]
whence equality holds in \eqref{ssI2} whereas strict inequality holds 
in \eqref{IBI}. On the other hand, if $Z'$ is a union of
exceptional components of $Y$, then strict inequality holds in
\eqref{ssI2} whereas equality holds in \eqref{IBI}.

Assume now that neither $Z$ nor $Z'$ is a union of exceptional components 
of $Y$. Let $n$ (resp.~$n'$) be the number of connected components of $Z'$ (resp.~$Z$) which are exceptional components of $Y$. Let $Z_1$ (resp.~$Z_2$) be the subcurve of $Y$ obtaining by removing from (resp.~adding to) $Z$ all 
the exceptional components $E$ of $Y$ intersecting $Z'$ (resp.~$Z$) at 
exactly 1 or 2 points. Then $Z_1$ and $Z_2$ are proper subcurves of $Y$ 
such that
\[
k_{Z_1}+2n'=k_{Z_2}+2n=k_Z
\]
and 
\[
\deg_{Z_1}(\omega)=\deg_Z(\omega)=\deg_{Z_2}(\omega).
\]
 Furthermore, 
\[
\deg_{Z_1}(\L)-n'\leq\deg_Z(\I)\text{ and }\deg_{Z_2}(\I)-n\leq\deg_Z(\L).
\]
So \eqref{ssI2} holds for $Z$ replaced by $Z_2$ only if
\eqref{IBI} 
holds, whereas \eqref{IBI} holds for $Z$ replaced by $Z_1$ 
only if \eqref{ssI2} holds. Furthermore, equality holds in 
\eqref{IBI} only if equality holds in \eqref{ssI2} for $Z$
replaced by $Z_2$. Since $Z_2$ contains some $Y_i$, this is not possible
if $\I$ is $Y_i$-quasistable for every $i=1,\dots,p$.  Also, equality
holds in \eqref{ssI2} only if it holds in \eqref{IBI} for $Z$
replaced by $Z_1$. Since $Z'_1$ is not a union of exceptional
components of $Y$, this is not possible if $\L$ is stably
balanced. 
\end{proof}

\begin{Thm}\label{isofcor}
The isomorphism of functors $\Phi$ of Theorem \ref{isof} restricts 
to isomorphisms of functors
$$
\Phi^b\:\P^b_{d,g}\lra\Jcal^{ss}_{d,g}\quad\text{and}\quad
\Phi^{sb}\:\P^{sb}_{d,g}\lra\Jcal^{s}_{d,g}.
$$
\end{Thm}

\begin{proof} Just combine Theorem \ref{isof} with Proposition \ref{biss}.
\end{proof}

\vskip0.5cm 

{\smallsc Instituto Nacional de Matem\'atica Pura e Aplicada, 
Estrada Dona Castorina 110, 22460--320 Rio de Janeiro RJ, Brazil}

{\smallsl E-mail address: \small\verb?esteves@impa.br?}

\vskip0.5cm

{\smallsc Universidade Federal Fluminense, Instituto de Matem\'{a}tica,
Rua M\'{a}rio Santos Braga, s/n, Valonguinho, 24020--005 Niter\'{o}i RJ, 
Brazil}

{\smallsl E-mail address: \small\verb?pacini@impa.br?}

\end{document}